\documentclass[11pt,a4paper,oneside,reqno]{amsart}

\usepackage[all]{xy}
\SelectTips{cm}{}  \calclayout
\usepackage{amsfonts,amssymb,amscd,amsmath,enumerate,verbatim,calc}

\usepackage{amscd,amssymb,amsopn,amsmath,amsthm,graphics,amsfonts,enumerate,verbatim,calc}
\usepackage[dvips]{graphicx}
\usepackage[colorlinks=true,linkcolor=blue,citecolor=blue]{hyperref}
\input xy
\xyoption{all}

\newcommand{\rt}{\rightarrow}
\newcommand{\lrt}{\longrightarrow}

\newcommand{\st}{\stackrel}
\newcommand{\pa}{\partial}

\newcommand{\C}{\mathbf{C} }

\newcommand{\K}{\mathbf{K} }
\newcommand{\X}{\mathbf{X} }
\newcommand{\Y}{\mathbf{Y} }

\newcommand{\F}{\mathbf{F} }

\newcommand{\Z}{\mathbb{Z} }

\newcommand{\CE}{\mathcal{E}}
\newcommand{\CA}{\mathcal{A} }

\newcommand{\CF}{\mathcal{F} }
\newcommand{\CG}{\mathcal{G} }

\newcommand{\CL}{\mathcal{L} }
\newcommand{\CH}{\mathcal{H} }

\newcommand{\CK}{\mathcal{K} }

\newcommand{\CQ}{\mathcal{Q} }

\newcommand{\CX}{\mathcal{X} }
\newcommand{\CY}{\mathcal{Y} }

\newcommand{\CZ}{\mathcal{Z} }
\newcommand{\CJ}{\mathcal{J} }

\newcommand{\im}{{\rm{Im}}}

\newcommand{\Ker}{{\rm{Ker}}}

\newcommand{\Hom}{{\rm{Hom}}}

\newtheorem{theorem}{Theorem}[section]

\newtheorem{lemma}[theorem]{Lemma}
\newtheorem{proposition}[theorem]{Proposition}

\theoremstyle{definition}

\newtheorem{example}[theorem]{Example}

\theoremstyle{plain}
\newtheorem{stheorem}{Theorem}[subsection]

\theoremstyle{definition}
\newtheorem{sdefinition}[stheorem]{Definition}

\numberwithin{equation}{section}
\begin{document}

\title[Purity and flatness in symmetric monoidal closed  exact categories]{Purity and flatness  in symmetric monoidal closed  exact  categories}
\author[ E. Hosseini and A. Zaghian]{E. Hosseini and A. Zaghian}

\address{Department of Mathematics, Shahid Chamran University of
Ahvaz, P.O.Box: 61357-83151, Ahvaz, Iran.}
\email{e.hosseini@scu.ac.ir}
\address{Department of Mathematics and Cryptography, Malek Ashtar University of
Technology, P.O.Box: 115-83145, Isfahan, Iran.}
\email{a.zaghian@mut-es.ac.ir}

\keywords{Closed category, exact category, flat object, pure injective object.\\
2010 Mathematical subject classification: 18G55, 18A05, 18A20,
18F20.}
\begin{abstract}
Let $(\CA, \textmd{-}\otimes\textmd{-})$ be a  symmetric monoidal
closed exact category. This category is a natural framework to
define the notions of $\otimes$-purity and $\otimes$-flatness. We
show that an object $\CF$ in $\CA$ is $\otimes$-flat if and only if
any conflation  ending in $\CF$ is $\otimes$-pure. Furthermore, we
prove a generalization of the Lambek Theorem (\cite{La63}) in $\CA$.
In the case $\CA$ is a quasi-abelian category, we prove that $\CA$
has enough pure injective objects.
\end{abstract}
\maketitle

\section{Introduction}
Let $\CA$ be an additive category. A sequence
$\xymatrix@C0.7pc@R0.9pc{ \CX\ar[r]^f&\CY\ar[r]^g&\CZ}$ in $\CA$ is
said to be a \textit{conflation} if $f$ is the kernel of $g$ and $g$
is the cokernel of $f$. The map  $f$ is called an \textit{inflation}
and $g$ is called a \textit{deflation} (\cite[Appendix A]{Ke90}).
For a given class $\CE$ of conflations in $\CA$,  $(\CA, \CE)$ is
said to be  an \textit{exact} \textit{category} if the following
axioms hold.
\begin{itemize}
\item [(i)]  For any object $A\in\CA$, the identity morphism $1_A$ is an inflation.
\item [(ii)]  For any object $A\in\CA$, the identity morphism $1_A$ is a
deflation.
\item [(iii)] Deflations (resp. Inflations) are closed under composition.
\item [(iv)] The  pullback (resp. pushout) of a deflation (resp. inflation)  along an arbitrary morphism exists and yields a deflation (resp. inflation).

\end{itemize}
See \cite[Appendix A]{Ke90} for more details on exact categories.
Assume that  $(\CA, \textmd{-}\otimes\textmd{-})$ is a
\textit{symmetric} \textit{monoidal} \textit{closed}  exact
category. Then there is a bifunctor
${\CH}om_\CA(\textmd{-},\textmd{-}):\CA^{\mathrm{op}}\times\CA\lrt\CA$
such that for any object $\CG$ in $\CA$ the functor -$\otimes
\CG:\CA\lrt\CA$ has a right adjoint
${\CH}om_{\CA}(\CG,\textmd{-}):\CA\lrt\CA$, i.e. for any pair of
objects $\CF$ and $\CK$ in $\CA$, we have  an isomorphism
\begin{align}\label{0}
  \mathrm{Hom}_{\CA}(\CF\otimes\CG,\CK)\cong
\mathrm{Hom}_{\CA}(\CF,{\CH}om_{\CA}(\CG,\CK))
\end{align}
which is naturally in all three arguments (see  \cite{KM71} for more
details). The bifunctor ${\CH}om_\CA(\textmd{-},\textmd{-})$ is
called the \textit{internal} \textit{hom} on $\CA$.

In some situations, the category $\CA$ does not have enough
projective objects (\cite[Ex III.6.2]{H}). This causes that,some of
the most important theorems in homological algebra do not hold in
general case. The Lazard-Govorov Theorem is one of them. It asserts
that any flat module over a ring  is a direct limit of finitely
generated free modules (\cite{La99} and \cite{Go65}). This theorem
has a significance role in the proofs of some important theorems in
homological algebra, especially \cite[Theorem 3]{Sten68} and
\cite[Theorem 2.4]{EG}. In this work we prove a generalization of
the main result of \cite{La63}, \cite[Theorem 3]{Sten68} and
\cite[Theorem 2.4]{EG} in $\CA$ regardless of whether the
Lazard-Govorov Theorem holds in $\CA$  or not (see also
\cite{Em16}). To this end, we need to prove that for any object
$\CF$ in $\CA$, $\CF^+= {\CH}om_{\CA}(\CF,\CJ)$ is a pure injective
object where $\CJ$ is an injective cogenerator for $\CA$. This
approach is very beneficial. Because we can generalize it to the
category of complexes in $\CA$ and deduce the same results for
complexes.

\vspace{.5cm}

\textbf{Setup:} In this work, all categories are symmetric monoidal
closed unless otherwise  specified.

\vspace{.5cm}

\section{Purity and flatness  in $\CA$}
This section is devoted to the relation between flatness and purity
in $\CA$.  A conflation
$$\xymatrix@C-0.9pc{
\CL:\CG'\ar[r]&\CG\ar[r]&\CG''}$$ in $\CA$ is said to be
\textit{pure} if for any object $\CG$ in $\CA$, $\CL\otimes\CG$ is
also a conflation. An object $\CJ$ in $\CA$ is called
\textit{injective} if for any conflation $\CL$ in $\CA$,
$\Hom_{\CA}(\CL,\CJ)$ is a conflation in the ordinary abelian exact
structure of abelian groups. Moreover $\CJ$ is an injective
cogenerator  if for any sequence $\xymatrix@C-0.9pc{
\CL:\CG'\ar[r]&\CG\ar[r]&\CG''}$ in $\CA$ where
$\Hom_{\CA}(\CL,\CJ)$ is a conflation of abelian groups then $\CL$
is a conflation in $\CA$. Let $\CJ$ be an injective cogenerator for
$\CA$ and
$(\textmd{-})^+:={\CH}om_{\CA}(\textmd{-},\CJ):\CA\lrt\CA$. We use
the contravariant functor $(\textmd{-})^+$ and prove  the following
important interpretation of purity in $\CA$.
\begin{proposition}\label{1}
A conflation  $\CL:\xymatrix@C-0.9pc{ \CG'\ar[r]&\CG\ar[r]&\CG''}$
in $\CA$ is pure if and only if $\mathcal{L}^+$ splits.

\end{proposition}
\begin{proof}
If $\CL$ is pure then, $(\CG')^+\otimes\CL$ is a conflation  and so,
Hom$_{\CA}((\CG')^+\otimes\CL,\CJ)$ is a conflation of abelian
groups. By \eqref{0}, we have an isomorphism
$$ \mathrm{Hom}_{\CA}((\CG')^+\otimes\CL,\CJ)\cong
\mathrm{Hom}_{\CA}((\CG')^+,{\CH}om_{\CA}(\CL,\CJ))$$  of
conflations of abelian groups. This implies that $\CL^+$ splits.
Conversely, assume that $\CL^+$ splits. Then for any object $\CF$ in
$\CA$, we have a conflation $\mathrm{Hom}_{\CA}(\CF,
{\CH}om_{\CA}(\CL,\CJ))$ of abelian groups. By \eqref{0}, there is
an isomorphism
$$\mathrm{Hom}_{\CA}(\CF,
{\CH}om_{\CA}(\CL,\CJ))\cong\mathrm{Hom}_{\CA}(\CF\otimes\CL,
\CJ)$$of conflations of abelian groups. Since $\CJ$ is an injective
cogenerator then, $\CF\otimes\CL$ is a conflation  and so we are
done.
\end{proof}
In the next lemma we show the existence of pure injective objects in
$\CA$. An object $\CE$ in $\CA$ is called \textit{pure}
\textit{injective} if it is injective  with respect to pure
conflations.

\begin{lemma}\label{2001}
For any object $\CX$ in $\CA$, $\CX^+$ is  pure injective.
\end{lemma}
\begin{proof}
Let $\CL$ be a pure conflation in $\CA$. By  \eqref{0}, we have the
isomorphism,
$$\mathrm{Hom}_{\CA}(\CL,
\CX^+)\cong\mathrm{Hom}_{\CA}(\CL\otimes\CX, \CJ)$$ of conflations
of abelian groups. It follows that, $\CX^+$ is pure injective.
\end{proof}

An object $\CF$ in $\CA$ is called \textit{flat} if
-$\otimes\CF:\CA\lrt\CA$ preserves conflations. By the proof of
Lemma \ref{2001}, for a given object $\CX$ in $\CA$, $\CX^+$ is pure
injective. Now, the conditions for proving the main theorem of the
article are available.
\begin{theorem}\label{elh20}
The following conditions are equivalent.

\begin{itemize}
\item[(i)]  $\CF$ is a flat object.
\item[(ii)] $\CF^+$ is injective.
\item[(iii)] Any conflation ending in $\CF$ is pure.

\end{itemize}

\end{theorem}

\begin{proof}
$(i)\Rightarrow(ii)$ Let $\CF$ be a flat object and $\CL$ be a
conflation in $\CA$. Then, $\CL\otimes\CF$ is also a conflation.
Apply $\mathrm{Hom}_{\CA}(\textmd{-},\CJ)$ and use the adjoint
property of -$\otimes\CF$ and ${\CH}om_{\CA}(\CF,\textmd{-})$ to
deduce the isomorphism
$$\mathrm{Hom}_{\CA}(\CL\otimes\CF,\CJ)\cong\mathrm{Hom}_{\CA}(\CL,\CF^+)$$
of  conflations of abelian groups. This shows the injectivity of
$\CF^+$.

$(ii)\Rightarrow(i)$ Let $\CF^+$ be an injective object.  For a
given conflation $\CL$ in $\CA$, $\mathrm{Hom}_{\CA}(\CL,\CF^+)$ is
a conflation of abelian groups and so, by the adjoint property of
-$\otimes\CF$ and ${\CH}om_{\CA}(\CF,\textmd{-})$,
$\mathrm{Hom}_{\CA}(\CL\otimes\CF,\CJ)$  is a conflation of abelian
groups. Since $\CJ$ is an injective cogenerator, then
$\CL\otimes\CF$ is a conflation. This shows the flatness of $\CF$.

$(i)\Rightarrow(iii)$. By  $(i)\Leftrightarrow(ii)$, $\CF^+$ is an
injective object. So, for a given conflation
$$\xymatrix@C-0.7pc@R-0.9pc{\CL:\CK\ar[r]& \CG\ar[r]&
\CF}$$ in $\CA$, $\CL^+$  splits and hence by Proposition \ref{1},
$\CL$ is pure.

$(iii)\Rightarrow(i)$.  By  $(i)\Leftrightarrow(ii)$, it is enough
to show that $\CF^+$ is injective. Let
\begin{align}\label{mess0}
 \xymatrix@C-0.7pc@R-0.9pc{\CF^+\ar[r]^f& \CG\ar[r]& \CK}
\end{align}
be a conflation in $\CA$. By the axioms of a closed symmetric
category, we have a morphism
$d_\CF:\CF\lrt{\CH}om_{\CA}(\CF^+,\CF\otimes\CF^+)$ and so by
\eqref{0}, there is  a morphism $\lambda_\CF:\CF\lrt\CF^{++}$ in
$\CA$ (see \cite[pp 97-99]{KM71}). This implies that the composition
$\CF^+\lrt\CF^{+++}\lrt\CF^+$ is the identity $1_{\CF^+}$. By the
axiom of an exact category, the top row of the following pullback
diagram
\[\xymatrix@C-0.7pc@R-0.9pc{\CK^+\ar[r]\ar@{=}[d]&\CQ\ar[r]^t\ar[d]^g&\CF\ar[d]^i\\
\CK^+\ar[r]&\CG^+\ar[r]^{f^+}&\CF^{++}}\] is a pure conflation and
hence, it is split ($\CK^+$ is pure injective). Then, we have a
morphism $t':\CF\lrt \CQ$ such that $tt'=1_{\CF}$. If
$g_1=gt':\CF\lrt \CG^+$, then $f^+g_1=f^+gt'=itt'=i$. Consequently,
in the following commutative diagram

\[\xymatrix@C-0.7pc@R-0.9pc{\CF^+\ar[r]^f\ar[d]^{j}&\CG\ar[r]\ar[d]^k&\CK\ar[d]\\
\CF^{+++}\ar[r]^{f^{++}}&\CG^{++}\ar[r]&\CK^{++} }\] $g_1^+kf =
g_1^+f^{++}j = i^+j = 1_{\CF^+}$. It follows that $\CF^+$ is an
injective object.

\end{proof}

\subsection{ Purity and flatness in the category of complexes in $\CA$.}
Recall that a complex in $\CA$ is a cochain
$$\X:\cdots \rt \CX^{n-1} \st{\pa_\X^{n-1}}{\rt} \CX^n \st{\pa_\X^{n}}{\rt}
\CX^{n+1} \rt \cdots$$ in $\CA$ such that for any $n\in\Z$,
$\pa_\X^{n}\pa_\X^{n-1}=0$. The category of all complexes in $\CA$
is denoted by $\C(\CA)$. A complex $\X$ in $\CA$ is called
\textit{acyclic} if for any $n\in\Z$,
$$\xymatrix@C-0.7pc@R-0.9pc{\Ker\pa_\X^n\ar[r]& \CX^n\ar[r]&
\im\pa_\X^n}$$ is a conflation in $\CA$ and it is called
\textit{pure} \textit{acyclic} if for any object $\CG$ in $\CA$,
$\X\otimes\CG$ is acyclic. An acyclic complex
$$\F:\cdots \rt \CF^{n-1} \st{\pa_\F^{n-1}}{\rt} \CF^n \st{\pa_F^{n}}{\rt}
\CF^{n+1} \rt \cdots$$ in $\CA$ is said to be \textit{flat} if for
any $n\in\Z$, $\Ker\pa_\F^n$ is a flat object in $\CA$. The exact
structure on $\CA$ induces an exact structure on $\CA$ as follows. A
sequence $\xymatrix@C0.7pc@R0.9pc{ \X\ar[r]^f&\Y\ar[r]^g&\K}$ in
$\C(\CA)$ is a conflation if for any $n\in\Z$,
$$\xymatrix@C0.7pc@R0.9pc{
\X^n\ar[r]^{f^n}&\Y^n\ar[r]^{g^n}&\K^n}$$is a conflation in $\CA$.
Proposition \ref{1}  will enable us to define a notion of purity in
$\C(\CA)$ and prove a generalization of \cite[Theorem 2.4]{EG} in
$\C(\CA)$.

\begin{sdefinition}\label{jimi}
A conflation $\mathbf{L}$ in $\C(\CA)$ is called pure if
$\mathbf{L}^+$ splits.
\end{sdefinition}

Notice that if  $\CA$ is locally finitely presented Grothendieck
category with enough projective objects, then the following result
and \cite[Theorem 2.4]{EG} are equivalent (see also \cite{Sten68}).
\begin{stheorem}\label{110}
The following conditions are equivalent.

\begin{itemize}
\item[(i)]  $\F$ is a flat complex in $\CA$.
\item[(ii)] $\F^+$ is an injective complex in $\CA$.
\item[(iii)] $\F$ is a pure acyclic complex of flat objects in $\CA$.
\item[(iv)] Any conflation ending in $\F$ is pure.

\end{itemize}

\end{stheorem}
\begin{proof}
The proof is straightforward.
\end{proof}

\subsection{Pure injective objects}
Pure injective objects are one of the most important generalizations
of injective objects which has a significance  role in homological
algebra. For instance, they are essential tools in the Swan's
approach on Cup products, derived functors and Hochschild cohomology
({\cite{Sw99}).

Our motivation on this subsection is a question asked by Rosicky in
\cite[Question 1]{Ro09} (see also \cite[Theorem 2.1]{Sw99}). This
question concerning about the existence of enough $\lambda$-pure
injective objects in a locally $\lambda$-presentable additive
category ($\lambda$ is an infinite regular cardinal). It is known
that there is another notion of purity which is different from the
$\lambda$-purity. This purity is known as $\otimes$-purity and
defined in monoidal categories. We are interested to ask
\cite[Question 1]{Ro09} for this purity and find an answer for it.
In this subsection, we show that any symmetric monoidal closed
quasi-abelian category has enough $\otimes$-pure injective objects.

Assume that $\CA$ is a  pre-abelian category, that is, an additive
category with kernels and cokernels (see \cite{RW77} and \cite{SW11}
for more details). We know that $\CA$ has a natural structure of an
exact category where conflations are short exact sequences. A
subobject $\CF$ of an object $\CG$ in $\CA$ is called \textit{pure}
if the canonical exact sequence
$$\xymatrix@C-0.7pc{0\ar[r]&\CF\ar[r]&\CG\ar[r]&\CG/\CF\ar[r]&0 }$$ is
pure in $\CA$.
\begin{stheorem}\label{2}
The category $\CA$ has enough pure injective objects.
\end{stheorem}
\begin{proof}
By the axioms of a symmetric monoidal closed category, there is a
morphism $d_\CF:\CF\lrt{\CH}om_{\CA}(\CF^+,\CF\otimes\CF^+)$ and so
by \eqref{0}, we have   a morphism $\lambda_\CF:\CF\lrt\CF^{++}$
 in $\CA$ (see \cite[pp 97-99]{KM71}). We show that $\lambda_\CF:\CF\lrt\CF^{++}$ is a pure monomorphism.
Let $\CK=\Ker\lambda_\CF$. Then we have the
following commutative diagram\[\xymatrix@C-0.7pc@R-0.9pc{0\ar[r]&\CK\ar[r]^i\ar[d]^{\lambda_\CK}&\CF\ar[d]^{\lambda_\CF}&\\
0\ar[r]&\CK^{++}\ar[r]^{i^{++}}&\CF^{++}}\]with exact rows. Since
$i^{++}$ is a monomorphism then $\lambda_{\CK}=0$. This implies that
$\CK=0$ and so $\lambda_\CF$ is a monomorphism. By Proposition
\ref{1}, it is enough to show that the epimorphism
$(\lambda_\CF)^+:(\CF^{++})^+\lrt\CF^+\lrt 0$ admits a section.
Since  $(\CF^{++})^+\cong(\CF^+)^{++}$ then we have the
following commutative diagram\[\xymatrix{\CF^+\ar[r]^{\lambda_{\CF^+}}\ar[dr]_{\mathrm{id}_{\CF^+}}&\CF^{+++}\ar[d]^{(\lambda_\CF)^+}&\\
&\CF^{+}}\]
 in $\CA$ (see \cite[pp.
100, diagram (1.3)]{KM71}). So, by Proposition \ref{1},
$\CF\lrt\CF^{++}$ is a pure monomorphism where $\CF^{++}$ is pure
injective by Lemma \ref{2001}.
\end{proof}

This theorem  gives another proof for \cite[Theorem 2.1]{Sw99} and
\cite[Corollary 4.6, 4.8]{EEO16} and enable us to prove to prove the
existence  of pure injective preenvelope in $\CA$.

\begin{example}
The category can be replaced by any symmetric monoidal closed
Grothenidieck category. For example, the category of modules over an
associative ring, the category of sheaves over an arbitrary
topological space and the category of quasi--coherent sheaves over
an arbitrary scheme (see \cite{H} for the algebraic geometry
background).
\end{example}
For more examples on non-abelian categories see the
\cite{Me18,Me12}.
\begin{stheorem}
The category $\C(\CA)$ has enough pure injective objects.
\end{stheorem}
\begin{proof}
The proof is similar to the proof of Theorem \ref{2}.
\end{proof}


\begin{thebibliography}{9999}





\bibitem[EEO16]{EEO16} {\sc  E. Enochs, S. Estrada, S. Odaba\c{s}i,} {\sl  Pure injective and absolutely pure sheaves,}
Proc. Edin. Math. Soc.  {\bf 59},  (2016), 623-640.
\bibitem[EG98]{EG} {\sc  E. Enochs, J. R.  Garcia Rozas,} {\sl  Flat covers of complexes,} J. Algebra, {\bf 210}, (1998), 86-102.
\bibitem [Em16]{Em16}  {\sc I. Emmanouil,} {\sl On pure acyclic complexes,} J. Algebra, {\bf 465}, (2016), 190-213.

\bibitem [Go65]{Go65}  {\sc V. E. Govorov,} {\sl On flat modules (in Russian),} Sibirsk. Mat. Z.  {\bf 6} (1965), 300-304.
\bibitem[Har97]{H} {\sc R. Hartshorne,} {\sl Algebraic geometry,}  GTM,  {\bf
52}, Springer-Verlag, (1997).
\bibitem[KM71]{KM71}  {\sc G. M. Kelly, S. Maclane,} {\sl Coherence in closed categories,} J. Pur. App. Alg.   {\bf 1 (1)}, (1971), 97-140.
\bibitem[Ke90]{Ke90} {\sc B. Keller,} {\sl Chain complexes and stable categories,}
Manus. Math. {\bf 67}, (1990),  379-417.
\bibitem[La69]{La99} {\sc D. Lazard,} {\sl Autour de la platitude,}  Bull. Soc. Math. France {\bf
97},  (1969), 81-128.

\bibitem[La64]{La63} {\sc J. Lambek,} {\sl A module is flat if and only if its character  module is injective,}  Canad. Math. Bull. {\bf7},  (1964), 237-243.
\bibitem[Me12]{Me12} {\sc B. Mesablishvili,} {\sl Descent in monoidal categories,}
 Theory Appl. Categ. {\bf 27}, (2012), 210-221.



\bibitem[Me18]{Me18} {\sc B. Mesablishvili,} {\sl Effective descent morphisms for Banach modules,}
J. Algebra and Its App.  {\bf 17}, No. {\bf 05}, (2018),
1850092(1-6).

\bibitem[Ro09]{Ro09} {\sc J. Rosicky,} {\sl Generalized purity, definability
and Brown representability,} Some Trends in Algebra, Prague (2009).

\bibitem[RW77]{RW77} {\sc F. Richman and E. A. Walker,} {\sl Ext in pre-abelian categories,}
Pac. J. Math. {\bf 71 (2)}, (1977),  521-535.
\bibitem[Sten68]{Sten68}{\sc B. Stenstrom,} {\sl Purity in Functor Categories,}
J. Algebra {\bf 8}, (1968), 352-361.
\bibitem[Sw99]{Sw99}  {\sc R. G. Swan,} {\sl Cup products in sheaf cohomology, pure injectives and a substitute for
projective resolutions,} J. Pur. App. Alg.   {\bf 144}, (1999),
169-211.
\bibitem[SW11]{SW11} {\sc D. Sieg and S-A. Wegner,} {\sl Maximal exact structure on additive categories,}
Math. Nachr, {\bf 284 (16)}, (2011),  2093-2100.
\end{thebibliography}
\end{document}